\documentclass[12pt]{elsarticle}
\usepackage{graphicx}
\usepackage[usenames,dvipsnames,svgnames,table]{xcolor}
\usepackage{amssymb}
\usepackage{amsmath}
\usepackage{amsthm}
\usepackage{comment}

\newcommand{\CA}{\mbox{$\mathsf{CA}$}}

\newcommand{\PHF}{\mbox{$\mathsf{PHF}$}}
\newcommand{\PHFN}{\mbox{$\mathsf{PHFN}$}}
\newcommand{\PHHF}{\mbox{$\mathsf{PHHF}$}}
\newcommand{\HF}{\mbox{$\mathsf{HF}$}}
\newcommand{\DHF}{\mbox{$\mathsf{DHF}$}}

\newcommand{\DHHF}{\mbox{$\mathsf{DHHF}$}}
\newcommand{\HHF}{\mbox{$\mathsf{HHF}$}}

\newcommand{\A}{\mbox{$\mathsf{A}$}}
\newcommand{\B}{\mbox{$\mathsf{B}$}}
\newcommand{\C}{\mbox{$\mathsf{C}$}}

\newtheorem{theorem}{Theorem}
\newtheorem{lemma}[theorem]{Lemma}

\journal{Theoretical Computer Science}
\date{}

\begin{document}
	
	\begin{frontmatter}
		
		\title{Distributing Hash Families with Few Rows}

\author[asu]{Charles J. Colbourn\fnref{nsf}} \ead{charles.colbourn@asu.edu}
\author[asu]{Ryan E. Dougherty} \ead{redoughe@asu.edu}
\author[monash]{Daniel Horsley\fnref{arc}} \ead{daniel.horsley@monash.edu}
\fntext[nsf]{Research supported in part by the U.S. National Science Foundation under Grant No. 1421058.}
\fntext[arc]{Research supported in part by the Australian Research Council through grants DP150100506 and FT160100048.}

\address[asu]{Computing, Informatics, and Decision Systems Engineering,
Arizona State University,
P.O. Box 878809,
Tempe, AZ 85287,
U.S.A.}

\address[monash]{School of Mathematical Sciences, Monash University, Vic 3800, Australia}

\begin{abstract}
		Column replacement techniques for creating covering arrays rely on the construction of perfect and distributing hash families with few rows, having as many columns as possible for a specified number of symbols.
		To construct distributing hash families in which the number of rows is less than the strength, we examine a method due to Blackburn and extend it in three ways.
		First, the method is generalized from homogeneous hash families (in which every row has the same number of symbols) to heterogeneous ones.
		Second, the extension treats distributing hash families, in which only separation into a prescribed number of parts is required, rather than perfect hash families, in which columns must be completely separated.
		Third, the requirements on one of the main ingredients are relaxed to permit the use of a large class of distributing hash families, which we call fractal.
		Constructions for fractal  perfect and distributing hash families are  given, and applications to the construction of   perfect hash families of large strength are developed.
\end{abstract}
		
		\begin{keyword}
			distributing hash family, perfect hash family, fractal  hash family, covering.
		\end{keyword}
		
	\end{frontmatter}
	

\section{Introduction}

A {\em  hash family} $\HF(N; k, w)$ is an $N \times k$ array whose entries are from a set of cardinality $w$.
Each row can be viewed as a hash function from a domain of size $k$ (the column indices) to a range of size $w$ (the symbols)  \cite{mehlhorn-phf}.
More generally, a {\em heterogeneous  hash family}, $\HHF(N; k, (w_1,\dots,w_N))$, is an $N \times k$ array in which the $i$th row contains (at most) $w_i$ symbols for $1 \leq i \leq N$; in other words, the ranges of the $N$ hash functions can differ.
Often we write $(w_1,\dots,w_N)$ in exponential notation: $w_1^{u_1} \cdots w_c^{u_c}$ means that the $N = \sum_{i=1}^c u_i$ rows can be partitioned into classes, so that in the $i$th class there are $u_i$ rows each employing (at most) $w_i$ symbols.

Let $\A$ be a heterogeneous hash family $\HHF(N; k, (w_1,\dots,w_N))$ and let $\{C_1,\dots,C_p\}$ be a partition into $p$ classes (some possibly empty) of some $t$-set of columns of $\A$. A row $r$ of $\A$ {\em separates} this partition if, for any two columns $c$ and $c'$ of $\A$ that are in distinct classes of the partition, the symbols in cells  $(r,c)$ and $(r,c')$ are distinct. If every partition of $t$ columns of $\A$ into $p$ classes is separated by some row, then $\A$ is a {\em distributing heterogeneous hash family} \DHHF$(N;k,(w_1,\dots,w_N),t,p)$.

A \DHHF$(N;k,(w_1,\dots,w_N),t,t)$ is a {\em perfect heterogeneous  hash family} \PHHF$(N;k,(w_1,\dots,w_N),t)$, and a \DHF$(N;k,w,t,t)$ is a {\em perfect   hash family} \PHF$(N;k,w,t)$.
The {\em perfect hash family number} \PHFN$(k,w,t)$ is the smallest $N$ for which a \PHF$(N;k,w,t)$ exists.

Mehlhorn \cite{mehlhorn-phf} introduced perfect hash families as an efficient tool for compact storage and fast retrieval of frequently used information. Stinson, Trung, and Wei \cite{STW} established that perfect hash families can be used to construct separating systems, key distribution patterns, group testing algorithms, cover-free families, and secure frameproof codes; see also \cite{BonehShaw00,StinsonWei98}.
Perfect hash families have also been applied in broadcast encryption \cite{Chor00,phf-broadcast} and threshold cryptography \cite{phf-threshold}.
An older survey on $\PHF$s is given in \cite{Czech97};
for recent results on the existence of perfect hash families see \cite{CLrechf,MTphf,WCphf} and references therein.
Perfect hash families arise as ingredients in some recursive constructions for covering arrays; this is our motivation for considering them here.

Let $N$, $k$, $t$, and $v$ be positive integers, and let $\C$ be an $N \times k$ array with entries from an alphabet $\Sigma$ of size $v$. We say $\C$ is a {\em covering array} \CA$(N;t,k,v)$ of {\em strength} $t$ if, for every $t$-subset $C$ of columns of $\C$ and for every function $f: C \rightarrow \Sigma$, there is a row $r$ of $\C$ such that the symbol in cell $(r,c)$ is $f(c)$ for each $c \in C$.

Covering arrays can be made via a method called column replacement, using hash families to select columns.  The basic result of this type is:

\begin{theorem}\label{phfcr} {\rm \cite{BS,MTvT}}
\label{phf}
If a $\PHF(M;\ell,k,t)$ and a $\CA(N;t,k,v)$ both exist then a $\CA(MN;t,\ell,v)$ exists. Indeed, when the $\CA(N;t,k,v)$ has $\rho$ constant rows, a $\CA(\rho+M(N-\rho);t,\ell,v)$ exists.
\end{theorem}

Colbourn \cite{Cdistrib} generalized Theorem \ref{phfcr} to employ distributing hash families.

\begin{theorem}\label{general}
Let $k \geq \min(t,v)$.
Suppose that  there exist a \DHF$(M;\ell,k,t,\min(t,v))$ and a \CA$(N;t,k,v)$ having $\rho$ constant rows.
Then a \CA$(\rho+(N-\rho)M;t,\ell,v)$ exists.
\end{theorem}

Colbourn and Torres-Jim\'enez \cite{CTJ} improved upon Theorem \ref{general} in two ways:
judiciously choosing symbols on which to place the constant rows, and using heterogeneous hash families.

\begin{theorem}\label{hetgen}
Suppose that there exist
\begin{enumerate}
\item a \CA$(N_i;t,k_i,v)$ having $\rho_i$ constant rows and $k_i \geq t$ for $1 \leq i \leq c$, and
\item a \DHHF$(M;\ell,k_1^{u_1} \cdots k_c^{u_c},t,\min(t,v))$.
\end{enumerate}
Let $\chi = \max(0,v-\sum_{i=1}^c u_i(v-\rho_i))$.
Then a \CA$(\chi +\sum_{i=1}^c u_i (N_i - \rho_i) ;t,\ell,v)$ exists.
\end{theorem}

Effective applications of Theorems \ref{phfcr}, \ref{general}, and \ref{hetgen}  require that both the covering arrays and the hash families employed have `few' rows.
In this paper, we focus on the numbers of rows in the distributing hash families used.

\section{Linear Bounds on Numbers of Columns}

A method of Blackburn \cite{blackburn2000} establishes:

\begin{lemma}\label{easy}
For positive integers $a_1, \dots a_t$, set  $\tau = \prod_{i=1}^t a_i$, and  $b_i = \frac{\tau}{a_i}$ for $1 \leq i \leq t$.
A \PHHF$(t;\tau,(b_1,\dots,b_t),t)$ exists in which every set of $1 \leq \ell \leq t$ columns is separated in at least $t+1-\ell$ rows.
\end{lemma}
\begin{proof}
Form a $t \times \tau$ array $A$, indexing columns by $\{ 1 , \dots a_1 \} \times \cdots \times \{ 1 , \dots a_t \}$.
Form row $j$ by ensuring that two columns contain the same symbol if and only if their indices agree in all coordinates other than the $j$th coordinate. Suppose to the contrary that for some $1 \leq \ell \leq t$, at most $t-\ell$ rows separate the $\ell$ columns $c_1,\dots,c_\ell$.
Form an edge-coloured  graph $G$ on vertex set $\{c_1,\dots,c_\ell\}$; for each row $r$ that does not separate the $\ell$ columns, place an edge of colour $r$ between two column indices whose columns contain the same symbol in row $r$.
Then $G$ has $\ell$ vertices and at least $\ell$ edges each having a different colour.
So $G$ contains a cycle $(x_0,\dots,x_{s-1})$ for some $s \leq \ell$.
Suppose that edge $\{x_0,x_{s-1}\}$ has colour $r$.
Then the columns indexed by $x_0$ and $x_{s-1}$ are the same in all rows other than $r$ but differ in row $r$.
On the other hand, for $0 \leq i < s-1$, edge $\{x_i,x_{i+1}\}$ does not have colour $r$, and hence the columns agree in row $r$.
This is a contradiction because the columns indexed by $x_0$ and $x_{s-1}$ must both agree and disagree in row $r$.
\end{proof}

Lemma \ref{easy} produces a  $\PHF(t;a^t,a^{t-1},t)$ and hence a   $\DHF(t;a^t,a^{t-1},t,p)$ for every $a \geq 2$ and $t \geq p$.
Hence the maximum number of columns grows superlinearly in the number of symbols for \DHHF s with $t \geq p \geq 2$ whenever the number of rows is at least $t$.
We are primarily interested in cases where the number of rows is less than the strength $t$.
In these cases, the number of columns cannot exceed a linear function of the number of symbols:

\begin{lemma}\label{btvt}
Let $t \geq p \geq 2$ and $t > n$.  If a \DHHF$(n;k,(w_1, \dots, w_n) ,t,p)$ exists, $k \leq \sum_{i=1}^n w_i$.
\end{lemma}
\begin{proof}
We adapt an argument from \cite{BazrafshanTvT}.
Let ${\bf w} = (w_1, \dots, w_n)$.
When $p \geq 3$, a \DHHF$(n;k,{\bf w},t,p)$ is also a \DHHF$(n;k,{\bf w}, t,p-1)$, so consider a \DHHF$(n;k,{\bf w},t,2)$, say $\A$.
If $n=1$, any repetition in the single row prevents the array from being a \DHHF; hence $k \leq w_1$.
Otherwise choose a row having the fewest symbols, say without loss of generality the first row having $w_1$ symbols.
Choose $m \leq w_1$ columns so that in the first row, every symbol that occurs among the columns not chosen also appears among the columns chosen.
By deleting the first row, and the $m$ chosen columns, we obtain an $(n-1) \times (k-m)$ array $\B$ that we claim is a \DHHF$(n-1;k-m,(w_2, \dots, w_n),t-1,2)$. Provided this claim holds, the lemma follows by induction.

Suppose otherwise that there is a partition into 2 parts $\{C_1,C_2\}$ of some $(t-1)$-set of columns of $\B$ that is not separated by any row of $\B$. Let $\sigma$ be a symbol that appears in one of the columns in $C_1$ in the first row of $\A$, and let $c$ be one of the chosen columns that contains $\sigma$ in the first row of $\A$. Then no row of $\A$ separates the partition $\{C_1,C_2 \cup\{c\}\}$, a contradiction.
\end{proof}

Lemma \ref{btvt} can be often improved upon, by  adapting a method of Blackburn \cite{blackburn2000} for perfect hash families to treat \DHHF s when the number of parts is large enough.
Let $\A$ be a  \DHHF$(n;k,(w_1,\dots,w_n),n+d,p)$  with $d \geq 1$.
Call a cell a {\em singleton} if the symbol it contains does not occur anywhere else in its row.
Form an $n \times k$ matrix $\B$, the {\em singleton array} of $\A$, setting the entry in row $r$ and column $c$ equal to 1 if the cell $(r,c)$ of $\A$ is a singleton, and equal to 0 otherwise.

\begin{lemma}\label{blackcov}
Let $\B$ be the singleton array of a \DHHF$(n;k,(w_1,\dots,w_n),n+d,p)$, $\A$, with $p \geq d+1 \geq 2$.
Then for every $d$-set of  columns, $\{c_1, \dots, c_d\}$,  of $\B$, there is at least one row $r$ of $\B$ in which the entry in row $r$ and column $c_i$ equals $1$ for all $1 \leq i \leq d$.
\end{lemma}
\begin{proof}
Suppose to the contrary that $\A$ is a \DHHF$(n;k,(w_1,\dots,w_n),n+d,p)$ with $p \geq d+1$, $\B$ is its singleton array, and for columns  $C= \{c_1, \dots, c_d\}$, no row of $\B$ has the $d$ entries in these columns all equal to 1.
For each row $r = 1,\dots,n$, there is a column $c_r \in C$ so that the entry of $\A$ in row $r$ and column $c_r$ is not a singleton.
Then let $d_r \neq c_r$ be a column index so that in row $r$, the entries of $\A$ in columns $c_r$ and $d_r$ are the same.
Now we form a partition of at most $n+d$ columns of $\A$ into at most $d+1$ classes that is not separated by any row of $\A$.
First, for every $c \in C$ such that $c=c_r$ for some $r$, form a class containing just the column index $c$.
Next, form a class  $\{d_r: 1 \leq r \leq n\} \setminus \{c_r: 1 \leq r \leq n\}$.
It follows that $c_r$ and $d_r$ are in different classes for each $1 \leq r \leq n$, so no row accomplishes this separation.
Because $|C| = d$,
we have chosen a partition of at most $n+d$ columns of $\A$ into at most $d+1$ classes, and hence
we have the required contradiction.
\end{proof}

If one singleton from each column of a \DHHF$(n;k,(w_1,\dots,w_n),n+d,p)$ can be identified, then $k$ singletons are identified and the number of identified singletons in row $r$ is at most $w_r$ for $1 \leq r \leq n$. Using this argument it can be seen that Lemma \ref{blackcov} improves on Lemma \ref{btvt} when $p \geq d+1 \geq 2$.

For certain parameters, a stronger conclusion can be obtained via the following argument.
Form a multigraph $G$ on vertex set $\{ c_r, d_r : 1 \leq r \leq n\}$ with edges $\{ \{ c_r, d_r\} : 1 \leq r \leq n\}$.
When $G$ can be properly coloured with $\gamma$ colours, the array $\A$ cannot be a \DHHF$(n;k,v,n+d,\gamma)$.
When on the columns $\{c_1, \dots, c_d\}$ some rows contain multiple entries that are not singletons, we may be able to choose $\{ c_r, d_r\}$ for certain values of $r$ in more than one way, and hence choose $G$ so as to reduce the chromatic number of $G$.

\section{Fractal Hash Families}

Later we describe a construction for \DHHF s with a number of rows smaller than the strength, which uses ingredient hash families that are required to satisfy an additional constraint.
The hash families to be introduced always have a number of rows equal to the strength $t$.
Lemma \ref{easy} produces a  $\PHF(t;a^t,a^{t-1},t)$ whenever $a,t \geq 2$, so the number of columns grows faster than linearly in the number of symbols for a \DHF$(t;k,w,t,p)$ with $t \geq p \ge 2$.
However, the growth of the number of columns is limited.
\begin{theorem}{\rm \cite{NiuC16}}
If a \DHF$(t;k,w,t,p)$ exists then $k \leq w^2$.  Moreover, if $t \geq 4$, $k \leq w^2-w$.
\end{theorem}
Indeed the growth rate is less than quadratic asymptotically:
\begin{theorem}{\rm \cite{GeChongWang17}}
Let $t \geq 4$ and let $k(w)$ be the largest integer for which a  \DHHF$(t;k(w),w,t,p)$ exists.  Then $k(w)$ is $\mbox{o}(w^2)$.
\end{theorem}

A $\DHHF(t;k,(v_1, \dots,  v_t),t,p)$ is {\em fractal} if $t \leq 2$, or if, for each row $j$, deleting row $j$ yields a fractal $\DHHF(t-1;k, (v_1, \dots, v_{j-1}, v_{j+1} , \dots ,  v_t),t-1,\min(p,t-1))$.
Fractal \PHHF s are simply fractal \DHHF s with $p=t$.
A $\DHHF(t;k,(v_1, \dots,  v_n),t,p)$ is {\em $\alpha$-fractal} if it is fractal and at least $\alpha$ rows of the \DHHF\ contain all distinct symbols.
The following equivalence is straightforward:

\begin{lemma}\label{addfractal}
An array is an $\alpha$-fractal $\DHHF(t;k,(v_1, \dots,  v_t),t,p)$ with $\alpha \geq 1$ if and only if it has a row $r$ containing all distinct symbols and the remaining rows form an $(\alpha-1)$-fractal $\DHHF(t-1;k,(v_1, \dots,  v_{r-1}, v_{r+1}, \dots, v_t),t-1,\min(p,t-1))$.
\end{lemma}

One characterization of fractal \DHHF s follows:

\begin{lemma}\label{fractaldhhf}
A \DHHF$(t;k,(v_1,\dots,v_t),t,p)$ is fractal if and only if every partition of $\ell$ of its columns into $\min(p,\ell)$ classes is separated by at least $t+1-\ell$ rows.
\end{lemma}

\begin{proof}
Suppose that there is some set $S$ of $\ell$ columns with $1 \leq \ell \leq t$ and some partition $\{C_1, \dots, C_{\min(p,\ell)}\}$ of $S$ into $\min(p,\ell)$ classes, so that  exactly $\rho \leq t-\ell$ rows separate the classes. The $t-\rho \geq \ell$ remaining rows, say without loss of generality the first $t-\rho$ rows, do not form a \DHHF$(t-\rho; k,(v_1,\dots,v_{t-\rho}),t-\rho,\min(\ell,p))$ because none of the rows separates classes $C_1, \dots, C_{\min(p,\ell)}$.
So the \DHHF$(t;k,(v_1,\dots,v_t),t,p)$  cannot be  fractal.

In the other direction, if $\A$ is not a fractal \DHHF$(t;k,(v_1,\dots,v_t),t,p)$, then some set of $\ell$ rows with $2 \leq \ell \leq t$, say without loss of generality the first $\ell$ rows, must yield an array $\B$ that is not a \DHHF$(\ell;k,(v_1,\dots,v_\ell),\ell,\min(p,\ell))$.
Let $\{C_1, \dots, C_{\min(p,\ell)}\}$ be a partition  that is separated by no row of $\B$.
Then $\{C_1, \dots, C_{\min(p,\ell)}\}$ is separated in $\A$ by at most $t-\ell$ rows.
\end{proof}

\subsection{Fractal \DHHF s}

\begin{lemma}\label{dhhfmustbephhf}
Whenever $t < \binom{p+1}{2}$, every \DHHF$(t;k,(w_1,\dots,w_t),t,p)$ is a \PHHF$(t;k,(w_1,\dots,w_t),t)$
\end{lemma}
\begin{proof}
Let $\A$ be an \HHF$(t;k,(w_1,\dots,w_t))$ that is not a  \PHHF$(t;k,(w_1,\dots,w_t),t)$.
Choose columns $\{c_1,\dots,c_t\}$ not separated by any row of $\A$.
Form a multigraph $G$ with $t$ vertices, $\{c_1,\dots,c_t\}$; for each row, choose a pair of columns $c_i$ and $c_j$ having the same symbol in this row and add $\{c_i.c_j\}$ as an edge.
Because $G$ has $t$ edges and $t < \binom{p+1}{2}$ by assumption, $G$ has a proper colouring in $p$ colours.
(A simple greedy colouring establishes that if $p+1$ colours were needed, the number of edges must be at least $\sum_{i=1}^p i = \binom{p+1}{2}$.)
Let $\{ C_1, \dots, C_p\}$ be the colour classes of  a proper colouring in $p$ colours.
Then the partition $\{C_1, \dots, C_p\}$ of $t$ columns of $\A$ is not separated by any row of $\A$, so $\A$ is not a \DHHF$(t;k,(w_1,\dots,w_t),t,p)$.
\end{proof}

\begin{table}[htbp]
\begin{center}
\framebox{
\begin{tabular}{ccccc}
\multicolumn{5}{c}{\PHF(4;5,4,4)}\\
1&1&2&3&4\\
1&2&2&3&4\\
1&2&3&3&4\\
1&2&3&4&4\\
\end{tabular}}
\framebox{
\begin{tabular}{cccccccccc}
\multicolumn{10}{c}{\DHF(4;10,4,4,2)}\\
1&1&1&2&2&2&3&3&3&4\\
1&2&3&1&2&3&1&2&3&4\\
1&2&3&2&3&1&3&1&2&4\\
1&2&3&3&1&2&2&3&1&4\\
\end{tabular}}
\end{center}
\caption{A \PHF(4;5,4,4) and a \DHF(4;10,4,4,2)}\label{phfdhfex}
\end{table}

By restricting the number of parts, fractal \DHHF s can exist with more columns than the corresponding fractal \PHHF s.
An example is given in Table \ref{phfdhfex}.
According to Niu and Cao \cite{NiuC16}, every \HF$(4;k,4)$ that accomplishes every separation of four columns into two classes of size two must have $k \leq 10$, and the array shown accomplishes every such separation with $k=10$.
One can verify that the array also separates all partitions of four columns into one class of size three and one of size one, and hence is a \DHF(4;10,4,4,2).
Lemma \ref{dhhfmustbephhf} ensures that every \DHF(4;$k$,4,4,3) must be a \PHF(4;$k$,4,4).
A simple counting argument ensures that a \PHF(4;$k$,4,4) has $k \leq 5$.
Hence while any \PHF(4;$k$,4,4) or \DHF(4;$k$,4,4,3) has $k \le 5$, restricting the number of classes in the partition to two doubles the number of columns possible.
We return to this in the concluding remarks.

We mention one construction of \DHHF s here:

\begin{lemma}
A fractal \DHHF$(3;a_1a_2,(a_1,a_1,a_2),3,2)$ exists whenever $a_1\geq a_2$ are positive integers.
\end{lemma}
\begin{proof}
Index columns by $\{0,\dots,a_1-1\} \times \{0,\dots,a_2-1\}$,
In column $(a,b)$, place $a$ in row 1, $b$ in row 3, and $a+b \pmod{a_1}$ in row 2.
\end{proof}

\subsection{Construction of fractal \PHHF s}

A  sufficient condition for a \PHHF\ to be fractal follows.

\begin{lemma}\label{singsing}
If a \PHHF$(t;k,(v_1,\dots,v_t),t)$ has at most one singleton in each row then it is fractal.
\end{lemma}

\begin{proof}
We prove the result by induction on $t$. The result is trivial when $t \leq 2$. Let $\A$ be a \PHHF$(t;k,(v_1,\dots,v_t),t)$ with $t \geq 3$ that has at most one singleton in each row. Let $\B$ be the array obtained from $\A$ by deleting an arbitrary row of $\A$, say the last without loss of generality. It suffices to show that $\B$ is a \PHHF$(t-1;k,(v_1,\dots,v_{t-1}),t-1)$, because then it will follow that $\B$ is fractal by our inductive hypothesis.

Suppose otherwise that there is a $(t-1)$-set $T$ of columns of $\B$ that is not separated by any row of $\B$. Since $t-1 \geq 2$ and there is at most one singleton in the last row of $\A$, there is a symbol $\sigma$ that, in the last row of $\A$, appears in some column in $T$ and also in some other column $c$ that may or may not be in $T$. Then $|T \cup \{c\}| \in \{t-1,t\}$ and no row of $\A$ separates the columns in $T \cup \{c\}$, a contradiction.
\end{proof}

Using Lemma \ref{singsing}, many \PHHF s can be seen to be fractal. For example, Walker and Colbourn \cite{WCphf} use a greedy construction of ``triangle-free, 3-regular, resolvable linear spaces (tfrrls)'' to produce many $\PHF(3; k, v, 3)$s having no singletons.
Fuji-Hara \cite{Fuji-Hara} gives an explicit construction of tfrrls using mutually disjoint spreads in a generalized quadrangle, thereby  proving that a $\PHF(3; q^2 (q+1), q^2, 3)$ exists when $q \ge 3$ is a prime power.
Using generalized quadrangles in Hermitian varieties, he also proved that a $\PHF(3; q^5, q^3, 3)$ exists for  $q$ a  prime power.
Lemma \ref{easy} produces a fractal $\PHF(t;a^t,a^{t-1},t)$; for $t=3$, Fuji-Hara's construction has many more columns, suggesting that the easy method of Lemma \ref{easy} is far from optimal.
See also \cite{ChongGe16} for further improvements when $t=3$ and when $t=4$.

Fractal \PHHF s can also be constructed recursively. The next result is based on \cite[Theorem~4.8]{WCphf}.

\begin{theorem}
Suppose that a \PHHF$(t;k,(v_1,\dots,v_t),t)$ exists with $k > t \geq 2$, and that $\ell$ is a positive integer.
Then a \PHHF$(t+1; \ell k, (\ell v_1,\dots,\ell v_t,k),t+1)$ exists.
If the \PHHF\ of strength $t$ is fractal, so is the \PHHF\ of strength $t+1$.
\end{theorem}

\begin{proof}
Let $\A_0,\ldots,\A_{\ell-1}$ be copies of the (fractal) \PHHF$(t;k,(v_1,\dots,v_t),t)$ with symbols renamed such that in each row the sets of symbols in $\A_i$ and $\A_j$ are disjoint when $i \neq j$. Let $\B$ be the $(t+1) \times \ell k$ array, with columns indexed by $\{0,\dots,k-1\} \times \{0,\dots,\ell-1\}$, obtained from $[\A_0 \cdots \A_{\ell-1} ]$ by appending a $(t+1)$st row that contains symbol $c$ in column $(c,s)$ for $0 \leq c < k$ and $0 \leq s \leq \ell-1$.

Let $T = \{ (c_i,s_i) : 1 \leq i \leq t+1\}$ be a set of $t+1$ column indices of $\B$.
If the coordinates $\{c_i : 1 \leq i \leq t+1\}$ are all distinct, then $T$ is separated in row $t+1$.
Otherwise $|\{c_i : 1 \leq i \leq t+1\}| \leq t$ and there is a row $r$ of $\A$  that separates the set $\{c_i : 1 \leq i \leq t+1\}$.
Because no columns $(c,s_i)$ and $(d,s_j)$ contain the same symbol unless $i = j$, $T$ is separated in row $r$ of $\B$.

Now we show that $\B$ is fractal when $\A$ is.
Suppose that $\C$ is obtained by deleting a row of $\B$.
If row $t+1$ is deleted, $\C$ is a fractal  \PHHF$(t; \ell k, (\ell v_1,\dots,\ell v_t),t)$ because $\A$ is fractal.
If row $i$ with $1 \leq i \leq t$ is deleted, then $\C$ is obtained by applying the construction of this lemma to the fractal \PHHF$(t-1;k,(v_1,\dots,v_{i-1}, v_{i+1}, \dots, v_t),t-1)$ obtained from $\A$ by deleting row $i$.
It therefore suffices to prove the statement when $t=2$ and this is routine to verify.
\end{proof}

\section{Blackburn's Method, revised}

To form a \DHHF\ with $n$ rows and strength $n+d$ with $1 \leq d < n$, we generalize a method due to Blackburn \cite{blackburn2000}.
In order to construct \PHF s,  he used the `easy' examples of fractal \PHF s from Lemma \ref{easy} without defining and using the fractal property explicitly.
In addition to fractal \DHHF s, we require a second  ingredient, as suggested by Lemma \ref{blackcov}.
An {\em $(n,m,d)$-covering} of {\em type} $(\rho_0,\dots,\rho_{m-1})$  is a collection of $n$ subsets $\{P_0, \dots , P_{n-1}\}$ of $\{0,\dots,m-1\}$ satisfying:
\begin{enumerate}
\item $| \{ P_r : 0 \leq r < n, P_r \ni c \}| = \rho_c$ for $0 \leq c < m$; and
\item For every $S \subseteq \{0,\dots,m-1\}$ with $|S|=d$,
$S$ is a subset of some set in $\{P_1,...,P_{n-1}\}$.
\end{enumerate}

\begin{theorem}\label{blackburn}
Suppose that there exist
\begin{itemize}
\item  an $(n,m,d)$-covering ${\mathcal P} = \{P_0, \dots , P_{n-1}\}$   of type $(\rho_0,\dots,\rho_{m-1})$, and
\item for each $0 \leq c < m$, a $\rho_c$-fractal \DHHF$(n;k_c,(v_{0,c}, \dots, v_{n-1,c}),n,p)$
in which, for $0 \leq r < n$, row $r$ contains all distinct symbols when $c \in P_r$.
\end{itemize}
Then there exists a $\DHHF(n; \sum_{c=0}^{m-1} k_c, (w_0, \dots , w_{n-1}), n+d,p')$ where \[ \begin{array}{rcl}
w_r & =  & \sum_{c=0}^{m-1} v_{r,c} \mbox{ for } 0 \leq r < n, \mbox{ and}\\
p' & = & \left \{ \begin{array}{lll} p & \mbox{if} & p<n-d \\
n+d & \mbox{if} & p\geq n-d.
\end{array}  \right.  \end{array} \]
\end{theorem}
\begin{proof}
Let $\A_c$ be the $\rho_c$-fractal \DHHF$(n;k_c,(v_{0,c}, \dots, v_{n-1,c}),n, p)$ for $0 \le c < m$.
Rename the symbols of each of $\{\A_c: 0 \leq c < m\}$ so that in each row the sets of symbols in $\A_i$ and $\A_j$ are disjoint when $i \neq j$.
Set $\B = [ \A_0 \cdots \A_{m-1} ]$.
Then $\B$ has $n$ rows and $\sum_{c=0}^{m-1} k_c$ columns, and for each $0 \leq r < n$, there are $w_r$ different symbols in row $r$.
To show that $\B$ is a  $\DHHF(n; \sum_{c=0}^{m-1} k_c, (w_0, \dots , w_{n-1}), n+d,p')$, it  suffices to show that every partition of $(n+d)$ columns into $p'$ classes is separated.

Consider a $(n+d)$-set $T$ of column indices and a partition $\mathcal{T}$ of $T$ into $p'$ classes.
For $0 \leq c < m$, let $\ell_c$ be the number of columns of $\A_c$ in $T$, and let $\mathcal{T}_c$ be the restriction of $\mathcal{T}$ to the columns of $\A_c$.
Because every two of the $\{\A_c: 0 \leq c < m\}$ share no symbols, it suffices to show that there is some fixed row of $\B$ that separates each partition $\mathcal{T}_c$ for $0 \leq c < m$.
Let $\pi$ be the largest element in $\{\ell_c: 0 \leq c < m\}$, let $L = \{ c : \ell_c \geq 2\}$, and let $\nu=|L|$.
Note that $\A_c$ trivially separates $\mathcal{T}_c$ for each $c \in \{0,\ldots,m\}\setminus L$.

If $\nu \leq d$, the $(n,m,d)$-covering contains a set $P_r$ with $L \subseteq P_r$.
Then in row $r$, for each $c \in L$, $\A_c$ contains all distinct symbols and therefore separates the partition $\mathcal{T}_c$.

So suppose that $\nu > d$.
Then, for each $0 \leq c < m$, $\ell_c \leq n-d$ and hence $\mathcal{T}_c$ has at most $\min(p',n-d) \leq p$ nonempty classes.
By Lemma \ref{fractaldhhf}, for each $c \in L$, $\mathcal{T}_c$ can fail to be separated in at most $\ell_c -1$  rows of $\A_c$ because $\A_c$ is a fractal \DHHF\ for $p$ parts.
Because $\nu > d$, $\sum_{h=0 }^{m-1} \max(0,\ell_h-1) \leq n+d-\nu < n$ and so at least one row of $\B$ separates each partition $\mathcal{T}_c$ for $0 \leq c < m$.
\end{proof}

We employ an easy variant of Theorem \ref{blackburn} repeatedly:

\begin{lemma}\label{varbb}
Suppose that $d \geq 1$ and a \PHHF$(n;\kappa,(w_1,\dots,w_n),n+d)$ exists.
\begin{enumerate}
    \item[\textup{(i)}]
Whenever $\alpha, k  \geq 1$, a \PHHF$(n+\alpha;\kappa+\alpha k,(w_1+\alpha k)^1 \cdots (w_n+\alpha k)^1 (\kappa + (\alpha-1)k+1)^{\alpha}, n+d+2\alpha)$ exists.
    \item[\textup{(ii)}]
In particular, whenever a \PHF$(n;\kappa,w,n+d)$ exists, a \PHF$(n+\alpha; \kappa  + \alpha (\kappa - w+1), \kappa + (\alpha-1) (\kappa-w+1)+1, n+d+2\alpha)$ exists.
\end{enumerate}
\end{lemma}

\begin{proof}
Statement (ii) follows from (i) by setting $w_1 = \cdots = w_n = w$ and  $k= \kappa + 1 - w$, so it suffices to prove (i).
Furthermore,  we only need to deal with the case where $\alpha=1$ because the remainder of the result  follows by induction.

Append a row with $\kappa$ distinct symbols to the \PHHF$(n;\kappa,(w_1,\dots,w_n),n+d)$ to form $\A_0$.
Form an $(n+1) \times k$ array $\A_1$ in which row $n+1$
contains $k$ occurrences of a single symbol, all other rows contain distinct symbols, and the sets of symbols in $\A_0$ and $\A_1$ are disjoint.
We claim that $\B = [\A_0 \A_1]$ is the required \PHHF.

Consider a $(n+d+2)$-set $T$ of column indices.
If $T$ contains at most one column of $\A_1$, then $T$ is separated by the last row of $\B$.
Otherwise, the restriction of $T$ to the columns of $\A_0$ contains at most $n+d$ columns and so is separated by some row $r$ of $\A_0$ other than the last.
Then row $r$ of $\B$ separates $T$.
\end{proof}

\section{Applications}

Between them, Theorem \ref{blackburn} and Lemma \ref{varbb} provide a  flexible framework for constructing \DHHF s.
In Lemmas \ref{dgen}--\ref{5-2} we give more concrete applications of these two results to producing \PHF s and \PHHF s with strength larger than their number of rows.
We conclude the section by considering the asymptotic ratio of columns to symbols in large \PHF s constructed by these lemmas.

We begin by choosing the covering in Theorem \ref{blackburn} to consist of all $d$-subsets of an $m$-set, an $( \binom{m}{d}, m, d)$-covering.

\begin{lemma}\label{dgen}
Let $m> d \geq 1$ be integers.
Suppose that a fractal
\[
\PHHF(\tbinom{m-1}{d};\kappa,(w_0,\dots,w_{\binom{m-1}{d}-1}),\tbinom{m-1}{d})
\]
exists.
Let $\sigma$ be the sum of the $m-d$ largest elements in $\{w_i:0\leq i \leq \binom{m-1}{d}-1\}$.
Then a \PHF$(\binom{m}{d};m\kappa,d\kappa + \sigma, \binom{m}{d}+d)$ exists.
\end{lemma}
\begin{proof}
Let $\mathsf{A}$ be the $\PHHF(\tbinom{m-1}{d};\kappa,(w_0,\dots,w_{\binom{m-1}{d}-1}),\tbinom{m-1}{d})$. Take the $(\binom{m}{d},m,d)$-covering $\{P_0, \dots, P_{\binom{m}{d}-1}\}$ in which the sets are all of the $d$-sets of $\{0,\ldots,m-1\}$.
This covering has $\rho_c = \binom{m-1}{d-1}$ for all $0 \leq c < m$.
Form a bipartite graph $G$ with vertex set $\{x_0, \dots, x_{\binom{m}{d}-1}\} \cup \{y_0, \dots, y_{m-1}\}$, placing an edge between $x_r$ and $y_c$ when the $r$th $d$-set  does not contain the element $c$. Note $\deg_G(x_i)=m-d$ for $0 \leq i < \binom{m}{d}$ and $\deg_G(y_i)=\binom{m-1}{d}$ for $0 \leq i < m$. So we can properly edge colour $G$ with $\binom{m-1}{d}$ colours $\{0,\dots,\binom{m-1}{d}-1\}$.
Now apply Theorem \ref{blackburn}. For each $0 \leq c <m$, use as an ingredient the $\binom{m-1}{d-1}$-fractal $\PHHF$ $\mathsf{A}_c$ obtained from $\mathsf{A}$ by adding $\binom{m-1}{d-1}$ rows of distinct symbols and rearranging the rows in such a way that, when edge $\{x_r,y_c\}$ of $G$ has colour $\ell$, row $\ell$ of $\mathsf{A}$ (with $w_\ell$ symbols) is row $r$ of $\mathsf{A}_c$. For $0 \leq r < \binom{m}{d}$, row $r$ of the resulting $\PHF$ has at most $d\kappa + \sigma$ symbols because $m-d$ distinct colours occur at the vertex $x_r$ of $G$.
\end{proof}

Two applications of  Lemma \ref{dgen}, with $d=1$ and $d=n-1$, are of particular interest.

\begin{lemma}\label{d=1}
When a fractal \PHHF$(n-1;\kappa,(w_1,\dots,w_{n-1}),n-1)$ exists, a \PHF$(n;n\kappa,\kappa+\sum_{i=1}^{n-1}w_i,n+1)$ exists.
\end{lemma}
\begin{proof}
Apply Lemma \ref{dgen} with $(m,d)=(n,1)$.
\end{proof}

\begin{lemma} \label{d=n-1}
For all $n \geq 2$ and $\kappa \geq 1$, a \PHF$(n;n\kappa,(n-1)\kappa+1,2n-1)$ exists.
\end{lemma}
\begin{proof}
Apply Lemma \ref{dgen} with $(m,d)=(n,n-1)$. (The \PHHF\ ingredient has one row and  strength 1.)
\end{proof}

Next we give other applications of Theorem \ref{blackburn} and Lemma \ref{varbb} to handle cases with $d \in \{n-2,n-3,n-4,n-5\}$.

\begin{lemma}\label{d=n-2}
Suppose that a \PHHF$(2;\kappa,(w_1,w_{2}),2)$ exists and  $n \geq 3$. Then
\begin{enumerate}
\item[\textup{(i)}] When $k \geq 1$, a \PHHF$(n;3\kappa+(n-3)k,(\kappa+(n-3)k+w_1+w_2)^3 (3\kappa+(n-4)k+1)^{n-3},2n-2)$ exists.
\item[\textup{(ii)}] When $w_1 + w_2 \leq 2\kappa$, a \PHF$(n;(2n-3)\kappa-(n-3)(w_1+w_2-1),(2n-5)\kappa -(n-4)(w_1+w_2-1)+1,2n-2)$ exists.
\end{enumerate}
\end{lemma}

\begin{proof}
It suffices to prove (i) because (ii) follows from (i) by setting $k = 2\kappa +1-w_1-w_2$. Lemma \ref{d=1} establishes (i) when $n=3$. Apply Lemma \ref{varbb}(i) with $\alpha = n-3$.
\end{proof}

\begin{lemma}\label{d=n-3}
Suppose that a \PHHF$(2;\kappa,(w_1,w_{2}),2)$ exists and  $n \geq 6$. Then
\begin{enumerate}
\item[\textup{(i)}] When $k \geq 1$, a \PHHF$(n;6\kappa+(n-6)k,(4\kappa+(n-6)k+w_1+w_2)^6 (6\kappa+(n-7)k+1)^{n-6},2n-3)$ exists.
\item[\textup{(ii)}] When $w_1 + w_2 \leq 2\kappa$, a \PHF$(n;(2n-6)\kappa-(n-6)(w_1+w_2-1),(2n-8)\kappa- (n-7)( w_1+w_2-1)+1,2n-3)$ exists.
\end{enumerate}
\end{lemma}

\begin{proof}
It suffices to prove (i) because (ii) follows from (i) by setting $k = 2\kappa +1-w_1-w_2$. By Lemma \ref{varbb}(i) with $\alpha=n-6$, it suffices to treat the case when $n=6$.
Form the 4-fractal  \PHHF s using the numbers of symbols in the columns given:
\begin{center}
\begin{tabular}{|c|c|c|c|c|c|}
$w_1$ & $w_2$ & $\kappa$& $\kappa$& $\kappa$& $\kappa$  \\
$\kappa$&$w_1$ & $w_2$ &  $\kappa$& $\kappa$& $\kappa$  \\
$w_2$ &$\kappa$&$w_1$ &   $\kappa$& $\kappa$& $\kappa$  \\
$\kappa$& $\kappa$& $\kappa$ & $w_1$ & $w_2$ & $\kappa$  \\
$\kappa$& $\kappa$& $\kappa$ &$\kappa$&$w_1$ & $w_2$   \\
$\kappa$& $\kappa$& $\kappa$ &$w_2$ &$\kappa$&$w_1$   \\
\end{tabular}
\end{center}
Let $P_0, \dots, P_5$ be the indices of the $\kappa$ entries in the six rows.
This yields  the $(6,6,3)$-covering.
Apply Theorem \ref{blackburn}.
\end{proof}

\begin{lemma}\label{d=n-4a}
Suppose that a fractal \PHHF$(3;\kappa,(w_1,w_{2},w_{3}),3)$ exists  and  $n \geq 6$. Then
\begin{enumerate}
\item[\textup{(i)}] When $k \geq 1$, a \PHHF$(n;6\kappa+(n-6)k,(3\kappa+(n-6)k+w_1+w_2+w_3)^6 (6\kappa+(n-7)k+1)^{n-6},2n-4)$ exists.
\item[\textup{(ii)}] When $w_1 + w_2+w_3 \leq \kappa$, a \PHF$(n;(3n-12)\kappa-(n-6)(w_1+w_2+w_3-1),(3n-15)\kappa-(n-7)( w_1+w_2+w_3-1)+1,2n-4)$ exists.
\end{enumerate}
\end{lemma}

\begin{proof}
It suffices to prove (i) because (ii) follows from (i) by setting $k = 3\kappa +1-w_1-w_2-w_3$. By Lemma \ref{varbb}(i) with $\alpha=n-6$ it suffices to treat the case when $n=6$.
We use a $(6,6,2)$-covering.
For $0 \leq j < 6$,  let $P_j = \{ j, j+1 \bmod{6}, j+3 \bmod{6}\}$.
To form the $3$-fractal \PHHF s $\{A_0, \dots, A_5\}$, set \[ v_{rc} = \left \{ \begin{array}{lll}
w_1 & \mbox{if} & 0 \leq r < 6, 0 \leq c < 6, c \equiv r+2 \pmod{6} \\
w_2 & \mbox{if} & 0 \leq r < 6, 0 \leq c < 6, c \equiv r+4 \pmod{6} \\
w_3 & \mbox{if} & 0 \leq r < 6, 0 \leq c < 6, c \equiv r+5 \pmod{6} \\
\kappa & \mbox{if} &  0 \leq c < 6, \mbox{ and }  c \equiv r, r+1, r+3 \pmod{6} \\
\end{array} \right .  \]
Apply Theorem \ref{blackburn}.
\end{proof}


\begin{lemma}\label{d=n-5}
Suppose that a fractal \PHHF$(4;\kappa,(w_1,w_{2},w_{3},w_{4}),4)$ exists  and  $n \geq 7$. Then
\begin{enumerate}
\item[\textup{(i)}] When $k \geq 1$, a \PHHF$(n;7\kappa+(n-7)k,(3\kappa+(n-7)k+w_1+w_2+w_3+w_4)^7 (7\kappa+(n-8)k+1)^{n-7},2n-5)$ exists.
\item[\textup{(ii)}] When $w_1 + w_2+w_3+w_4 \leq 4\kappa$, a   \PHF$(n;(4n-21)\kappa-(n-7)(w_1+w_2+w_3+w_4-1),(4n-25)\kappa- (n-8)( w_1+w_2+w_3+w_4-1)+1,2n-5)$ exists.
\end{enumerate}
\end{lemma}

\begin{proof}
It suffices to prove (i) because (ii) follows from (i) by setting $k = 4\kappa +1-w_1-w_2-w_3-w_4$. By Lemma \ref{varbb}(i) with $\alpha=n-7$ it suffices to treat the case when $n=7$.
We use a $(7,7,2)$-covering.
When $n=7$, for $0 \leq j < 7$,  let $P_j = \{ j, j+1 \bmod{7}, j+3 \bmod{7}\}$.
To form the $3$-fractal \PHHF s $\{A_0, \dots, A_{6}\}$, set \[ v_{rc} = \left \{ \begin{array}{lll}
w_1 & \mbox{if} & 0 \leq r < 7, 0 \leq c < 7, c \equiv r+2 \pmod{7} \\
w_2 & \mbox{if} & 0 \leq r < 7, 0 \leq c < 7, c \equiv r+4 \pmod{7} \\
w_3 & \mbox{if} & 0 \leq r < 7, 0 \leq c < 7, c \equiv r+5 \pmod{7} \\
w_4 & \mbox{if} & 0 \leq r < 7, 0 \leq c < 7, c \equiv r+6 \pmod{7} \\
\kappa & \mbox{if} &  0 \leq c < 7, \mbox{ and }  c \equiv r, r+1, r+3 \pmod{7} \\
\end{array} \right .  \]
Apply Theorem \ref{blackburn}.
\end{proof}

Finally we treat a special case with $d=2$.

\begin{lemma}\label{5-2}
If there exist
\begin{itemize}
\item a \PHHF$(2;\kappa_2,(v_{1,2},v_{2,2}),2)$,
\item a \PHHF$(2;\kappa_3,(v_{1,3},v_{2,3}),2)$, and
\item a fractal \PHHF$(3;\kappa_1,(v_{1,1},v_{2,1},v_{3,1}),3)$,
\end{itemize}
then
a \PHHF$(5;2\kappa_1+2\kappa_2+\kappa_3,(w_0,\dots,w_4)  ,7)$ exists with
$w_0  =  \kappa_1+2\kappa_2+v_{1,1}+v_{1,3}$,
$w_1  =  \kappa_1+2\kappa_2+v_{1,1}+v_{2,3}$,
$w_2  =  2\kappa_1+\kappa_3+v_{1,2}+v_{2,2}$,
$w_3  =  \kappa_2+\kappa_3+v_{2,1}+v_{3,1}+v_{2,2}$,
$w_4  =  \kappa_2+\kappa_3+v_{1,2}+v_{2,1}+v_{3,1}$.
\end{lemma}

\begin{proof}
Using a fractal \PHHF$(3;\kappa_1,(v_{1,1},v_{2,1},v_{3,1}),3)$,
a \PHHF$(2;\kappa_2,(v_{1,2},v_{2,2}),2)$, and a \PHHF$(2;\kappa_3,(v_{1,3},v_{2,3}),2)$,
form five \PHHF s on $5$ rows by placing the rows as indicated in each column shown; when $\kappa_i$ is specified, the row is all distinct symbols.

\begin{center}
\begin{tabular}{|c|c|c|c|c|}
$v_{1,1}$ & $\kappa_1$ & $\kappa_2$ &  $\kappa_2$ &  $v_{1,3}$ \\
$\kappa_1$ & $v_{1,1}$ &  $\kappa_2$ &   $\kappa_2$ &  $v_{2,3}$\\
$\kappa_1$ & $\kappa_1$ & $v_{1,2}$ & $v_{2,2}$ & $\kappa_3$ \\
$v_{2,1}$ & $v_{3,1}$ & $v_{2,2}$ &  $\kappa_2$ &  $\kappa_3$\\
$v_{3,1}$ &$v_{2,1}$&  $\kappa_2$ &  $v_{1,2}$ & $\kappa_3$\\
\end{tabular}
\end{center}

Let $P_0, \dots, P_4$ be the indices of the $\kappa$ entries in the five rows.
This yields  the $(5,5,2)$-covering.
Apply Theorem \ref{blackburn}.
\end{proof}

Numerous cases have been handled by Lemma \ref{dgen}.
We could take $m=5$ and $d=2$ to yield \PHF s with 10 rows and strength 12, or $m=5$ and $d=3$ to yield \PHF s with 10 rows and strength 13.
However, Lemma \ref{dgen} need not yield the best result asymptotically, as shown by the $(10,10,2)$-covering with blocks 0169, 2379, 4589, 0178, 2368, 4567, 024, 035, 125, 134.
Using ingredients with $\kappa$ columns on elements $\{0,\dots,5\}$, and $\kappa/2$ on elements $\{6,\dots,9\}$, the number of columns grows like $8\kappa$ while the number of symbols grows like $3\kappa$.

Table \ref{ptable} summarizes  the best asymptotic ratio of columns to symbols in large \PHF s constructed using the lemmas in this section; this extends somewhat a table from \cite{blackburn2000}.
\renewcommand{\tabcolsep}{3pt}

\begin{table}[htb]
\begin{center}
\begin{tabular}{c|cccccccccc}
$n \downarrow d \rightarrow$ & 1 & 2 & 3 & 4 & 5 & 6 & 7 & 8 & 9 \\
\hline
2 & \ref{d=1}: 2 &\\
3 & \ref{d=1}: 3 & \ref{d=n-1}: $\frac{3}{2}$\\
4 & \ref{d=1}: 4 & \ref{d=n-2}: $\frac{5}{3}$& \ref{d=n-1}: $\frac{4}{3}$\\
5 & \ref{d=1}: 5 & \ref{5-2}: $\frac{9}{5}$ & \ref{d=n-2}: $\frac{7}{5}$& \ref{d=n-1}: $\frac{5}{4}$\\
6 & \ref{d=1}: 6 & \ref{d=n-4a}: 2 & \ref{d=n-3}: $\frac{3}{2}$ & \ref{d=n-2}: $\frac{9}{7}$& \ref{d=n-1}: $\frac{6}{5}$\\
7 & \ref{d=1}: 7 &  \ref{d=n-5}: $\frac{7}{3}$ & \ref{d=n-4a}: $\frac{3}{2}$& \ref{d=n-3}: $\frac{4}{3}$& \ref{d=n-2}: $\frac{11}{9}$& \ref{d=n-1}: $\frac{7}{6}$\\
8 & \ref{d=1}: 8 & & \ref{d=n-5}: $\frac{11}{7}$ &  \ref{d=n-4a}: $\frac{4}{3}$& \ref{d=n-3}: $\frac{5}{4}$& \ref{d=n-2}: $\frac{13}{11}$& \ref{d=n-1}: $\frac{8}{7}$\\
9 & \ref{d=1}: 9 & & & \ref{d=n-5}: $\frac{15}{11}$ &  \ref{d=n-4a}: $\frac{5}{4}$& \ref{d=n-3}: $\frac{6}{5}$& \ref{d=n-2}: $\frac{15}{13}$& \ref{d=n-1}: $\frac{9}{8}$\\
10 & \ref{d=1}: 10 &  && & \ref{d=n-5}: $\frac{19}{15}$ &  \ref{d=n-4a}: $\frac{6}{5}$& \ref{d=n-3}: $\frac{7}{6}$& \ref{d=n-2}: $\frac{17}{15}$& \ref{d=n-1}: $\frac{10}{9}$ \\
\hline
\end{tabular}
\end{center}
\caption{\PHF s with few rows from Lemmas \ref{dgen}--\ref{5-2}. For each case, the number of the relevant lemma, and the asymptotic ratio of $\frac{k}{v}$ achieved, is given}\label{ptable}
\end{table}

\section{Existence tables}

In order to assess the impact of using Blackburn's construction for perfect hash families using fractal ingredients, we have created tables on the best-known upper bounds on  $\PHFN(k, v, t)$ for $k \le 10^9, v \le 2500$, and $3 \le t \le 11$ \cite{ryantables}.
These tables report on over 385,000 parameter situations.
Of those, 2,658 are improvements that result from the generalization of Blackburn's theorem.
Improvements were found only for larger strengths, in particular when $t \geq 6$.
We provide here a representative collection of improvements, restricting our attention to cases with  $N < t$ and $v \le 250$.
Each table considers a selection of $N$ and $t$;  then $k_{old}$ is the largest number of columns found without using the fractal version of the Blackburn construction, while $k_{fractal}$ gives the largest number of columns obtained using in addition fractal \PHHF s in the Blackburn construction.
In order to highlight the more significant improvements, we only report cases when $k_{fractal} \ge k_{old}+5$.

Naturally, other recursive constructions  can and do make further improvements, but we do not address them here.

\renewcommand{\arraycolsep}{5pt}
\renewcommand{\tabcolsep}{5pt}

\begin{table}[htbp]
	\centering
	\caption{$t$=6, $N=4$}
	\label{tbl:t_6.4}
\begin{tabular}{r|rrrrrrrr}
	$v$&121&127&163&166&169&211&217&\\
	$k_{fractal}$ &188&198&256&260&266&334&344&\\
	$k_{old}$&183&192&247&253&259&326&337&\\
\end{tabular}

	\centering
	\caption{$t$=6, $N=5$}
	\label{tbl:t_6.5}
\begin{tabular}{r|rrrrrrrrr}
\hline
	$v$&50&63&68&75&83&93&101&108&115\\
	$k_{fractal}$ &90&135&140&155&175&205&225&240&255\\
	$k_{old}$&82&104&113&125&139&172&197&216&223\\
\hline
	$v$&121&130&135&140&148&157&165&172&181\\
	$k_{fractal}$ &265&290&295&300&320&345&365&380&405\\
	$k_{old}$&229&254&259&264&272&281&292&303&313\\
\hline
	$v$&189&196&207&215&223&228&238&246\\
	$k_{fractal}$ &425&440&475&495&515&520&550&570\\
	$k_{old}$&405&412&423&431&439&444&454&481\\
\end{tabular}
\end{table}

\begin{table}[htbp]
	\centering
	\caption{$t$=7, $N=5$}
	\label{tbl:t_7}
\begin{tabular}{r|rrrrrrrrr}
\hline
	$v$&78&80&81&82&114&115&123&124&127\\
	$k_{fractal}$ &116&120&121&122&174&175&189&190&194\\
	$k_{old}$&111&114&115&117&169&170&183&185&189\\
\hline
	$v$&128&129&130&131&133&134&135&136&137\\
	$k_{fractal}$ &196&198&201&202&204&205&207&208&209\\
	$k_{old}$&190&191&192&193&195&196&197&198&199\\
\hline
	$v$&138&139&141&142&143&144&146&148&149\\
	$k_{fractal}$ &210&214&217&218&219&224&226&228&233\\
	$k_{old}$&200&202&204&207&211&216&211&213&217\\
\hline
	$v$&150&154&155&156&157&158&159&161&162\\
	$k_{fractal}$ &234&238&239&240&242&244&246&248&252\\
	$k_{old}$&222&226&227&228&229&230&231&237&242\\
\hline
	$v$&164&167&168&169&170&171&172&175&177\\
	$k_{fractal}$ &257&260&264&265&266&268&269&273&276\\
	$k_{old}$&248&251&252&253&254&255&256&259&261\\
\hline
	$v$&179&182&183&184&185&186&187&188&189\\
	$k_{fractal}$ &281&284&287&288&289&290&292&294&297\\
	$k_{old}$&266&269&270&272&274&275&277&279&281\\
\hline
	$v$&190&191&192&193&194&195&196&197&201\\
	$k_{fractal}$ &298&299&304&305&306&307&308&312&316\\
	$k_{old}$&283&285&287&289&290&281&292&293&308\\
\hline
	$v$&203&204&206&208&209&210&211&212&213\\
	$k_{fractal}$ &318&320&326&328&329&333&334&335&336\\
	$k_{old}$&313&314&316&318&318&320&321&322&323\\
\hline
	$v$&214&216&217&219&221&222&223&224&226\\
	$k_{fractal}$ &340&343&344&346&348&349&350&354&360\\
	$k_{old}$&324&326&327&329&331&333&335&337&341\\
\hline
	$v$&227&229&230&231&232&233&234&235&236\\
	$k_{fractal}$ &362&365&366&367&368&369&273&377&380\\
	$k_{old}$&343&347&349&351&353&355&357&359&369\\
\hline
	$v$&242&243&244&245&246&247&248&249\\
	$k_{fractal}$ &386&387&388&389&389&391&393&398\\
	$k_{old}$&372&375&376&378&380&382&384&386\\
\end{tabular}
\end{table}

\begin{table}[htbp]
	\centering
	\caption{$t$=8, $N=6$}
	\label{tbl:t_8}
\begin{tabular}{r|rrrrrrrr}
	$v$&108&120&135&158&174&184&195&207\\
	$k_{fractal}$ &162&192&216&240&270&288&300&324\\
	$k_{old}$&156&175&195&233&258&271&291&308\\
\hline
	$v$&218&227&240\\
	$k_{fractal}$ &336&360&384\\
	$k_{old}$&330&339&360\\
\end{tabular}
\end{table}

\begin{table}[htbp]
	\centering
	\caption{$t$=9, $N=6$}
	\label{tbl:t_9}
\begin{tabular}{r|rrrrrrrr}
	$v$&173&181&191&194&231&236&239\\
	$k_{fractal}$ &240&252&264&270&324&330&336\\
	$k_{old}$&234&244&256&259&315&320&323\\
\end{tabular}

	\centering
	\caption{$t$=10, $N=7$}
	\label{tbl:t_10}
\begin{tabular}{r|rrrrrrrr}
	$v$&191&215&239\\
	$k_{fractal}$ &215&298&323\\
	$k_{old}$&253&280&318\\
\end{tabular}

	\centering
	\caption{$t$=11, $N=7$}
	\label{tbl:t_11}
\begin{tabular}{r|rrrrrrrr}
	$v$&143&179&191&209&239\\
	$k_{fractal}$ &183&230&245&269&308\\
	$k_{old}$&178&224&239&264&300\\
\end{tabular}
\end{table}

\section{Concluding Remarks}

The use of fractal \DHHF s and \PHHF s in the Blackburn method leads to many constructions for \DHHF s with a number of rows less than the strength.
Our motivation for seeking these improvements has been to improve bounds for covering arrays.
Many  improvements are reported in the online covering array tables \cite{catables}.
We make no effort  to enumerate them here, contenting ourselves to mention a few illustrative examples.

Renaming symbols in each column of a covering array, we can always produce at least one constant row.
Then by Theorem \ref{hetgen}, the existence of a \PHF$(4;260,166,6)$ and a \CA$(N;6,166,v)$ ensures that a \CA$(4N-4;6,260,v)$ exists.
This yields the smallest covering array for these parameters when $v \in \{7,8,9,11,12,13\}$.
Because a  \PHHF$(2;ab,(a,b),2)$ exists whenever $a,b \geq 1$,  a \PHHF$(2;55,(7,8),2)$ exists.
Then using Lemma \ref{d=n-2}(i) with $k=95$, there is a \PHHF$(4;260,165^3 166^1,6)$.
Hence by Theorem \ref{hetgen}.  if a \CA$(N;6,165,v)$ and a \CA$(N';6,166,v)$ both exist, a \CA$(3N+N'-4;6,260,v)$ exists.
This illustrates how the use of heterogeneous hash families can reduce the number of rows in the covering array produced.

Using Lemma \ref{d=1} with a fractal \PHF$(4;81,25,4)$ (found by the method of \cite{PHFdens}) yields a \PHF$(5;405,181,6)$.
Then by Theorem \ref{hetgen}, the existence of  a \CA$(N;6,181,v)$ ensures that a \CA$(5N-5;6,405,v)$ exists.
This yields the smallest covering array for these parameters when $v \in \{5,7,8,9,11,13,18,19\}$.

Extending the Blackburn method to fractal and heterogeneous hash families therefore improves on known constructions for covering arrays even within the ranges currently tabulated at \cite{catables}.
To see that the extension to distributing hash families is also effective, we consider larger strengths.
We use the framework of Lemma \ref{d=n-5}, taking $\kappa = 10$.
According to \cite{PHFdens}, a \PHF$(4;10,6,4)$ exists, and it can be easily verified that one is fractal.
Then a \PHHF$(7;70,54,9)$, and hence a  \DHHF$(7;70,54,9,p)$,  exists for all $2 \leq p \leq 9$.
Using instead the \DHF(4;10,4,4,2) from Table \ref{phfdhfex} in the construction of Lemma \ref{d=n-5}, we produce a \DHF(7;70,46,9,2), using many fewer symbols.
When used in a column replacement strategy for covering arrays, this enables us to use a binary covering array with 46 columns rather than 54, which can be a substantial improvement.

\section*{Bibliography}

\def\cprime{$'$}

\end{document}